\documentclass[12pt,leqno,amscd,amssymb,verbatim, url]{amsart}
\usepackage{amsmath,amscd}
\usepackage{hyperref}
\usepackage{mathrsfs}
\let\oldsection=

\renewcommand{\subsection}[1]{\par\vspace{.18in}\noindent\addtocounter{subsection}{1}\setcounter{equation}{0}{\bf\thesubsection\hspace{9pt}#1}}

\usepackage[english]{babel}
\usepackage[latin1]{inputenc}
\usepackage{times}
\usepackage[T1]{fontenc}
\usepackage{amsthm, amsmath, amssymb}
\usepackage{tikz}
 
\usetikzlibrary{arrows,shapes}
\usetikzlibrary{decorations.pathmorphing}
\usetikzlibrary{calc}

\newtheorem{thm}{Theorem}[subsection]

\makeatletter\let\c@fact\c@theorem\makeatother\newtheorem{lem}[thm]{Lemma}
\newtheorem{cor}[thm]{Corollary}

\newtheorem{prop}[thm]{Proposition}

\theoremstyle{definition}

\newtheorem{conj}[thm]{Conjecture}

\newtheorem{rem}[thm]{Remark}

\numberwithin{equation}{subsection}
\numberwithin{thm}{section}
\newcommand{\irr}{\text{\rm Irr}}

\newcommand{\wH}{{\widetilde \sH}}

\newcommand{\wX}{{\widetilde{X}}}

\newcommand{\wsH}{{\widetilde{\mathcal H}}}

\newcommand{\IND}{\mbox{IND}}

\newcommand{\sH}{{\mathcal H}}
\newcommand{\sT}{{\mathcal T}}
\newcommand{\sA}{{\mathcal A}} 
\newcommand{\sX}{{\mathcal X}}

\newcommand{\sZ}{{\mathscr{Z}}}
\newcommand{\wsA}{{\widetilde{\mathcal A}}}

\newcommand{\sC}{{\mathscr {C}}}

\newcommand{\Ext}{{\text{\rm Ext}}}

\newcommand{\sG}{{\mathscr G}}
\newcommand{\sS}{{\mathcal S}}

\newcommand{\Amod}{A\mbox{--mod}}

\newcommand{\Hom}{\text{\rm Hom}}

\newcommand{\End}{\operatorname{End}}

\newcommand{\ind}{\operatorname{ind}}

\newcommand{\fN}{\operatorname{{\mathfrak{N}}}}

\newcommand{\sO}{{\mathscr{O}}}

\newcommand{\wDelta}{{\widetilde{\Delta}}}

\newcommand{\wP}{{\widetilde{P}}}

\newcommand{\wS}{{\widetilde{S}}}
\newcommand{\wT}{{\widetilde{\sT}}}

\newcommand{\wY}{{\widetilde{Y}}}

\newcommand{\blist}{\begin{list}{\rom{(\roman{enumi})}}{\setlength
{\leftmarg in}{0em} \setlength{\itemindent}{7ex}
\setlength{\labetwo-sidedlsep}{2ex}\setlength{\listparindent}{\parindent}
\usecounter{enumi}}}
\newcommand{\elist}{\end{list}}

\def\sJ{{\mathcal J}}

\def\sQ{{\mathscr Q}}

\def\sO{{\mathcal O}}
\def\la{{\lambda}}
\def\fN{{\mathfrak N}}
\def\fM{{\mathfrak M}}
\def\fL{{\mathfrak L}}
\def\sL{{\mathcal L}}
\def\bbC{{\mathbb C}}
\def\bbQ{{\mathbb Q}}
\def\sR{{\mathscr R}}
\def\wwT{{\widetilde T}}

\begin{document}

\begin{abstract}  The (Iwahori-)Hecke algebra in the title is a $q$-deformation $\sH$
of the group algebra of a finite Weyl group $W$.  The algebra $\sH$ has a natural enlargement
to an endomorphism algebra $\sA=\End_\sH(\sT)$ where $\sT$ is a $q$-permutation module. In type
$A_n$ (i.e., $W\cong {\mathfrak S}_{n+1}$), the algebra $\sA$ is a $q$-Schur algebra which is
quasi-hereditary and plays an important role in the modular representation of the
finite groups of Lie type.   In other types, $\sA$ is
not always quasi-hereditary, but the authors conjectured 20 year ago that $\sT$ can be
enlarged to an $\sH$-module $\sT^+$ so that $\sA^+=\End_\sH(\sT^+)$ is at least standardly stratified, a weaker condition than being quasi-hereditary, but with ``strata" corresponding to Kazhdan-Lusztig two-sided cells.

The main result of this paper is a ``local" version of this conjecture in the equal parameter case, viewing $\sH$ as  defined over
${\mathbb Z}[t,t^{-1}]$, with the localization at a prime ideal generated by a cyclotomic polynomial  
$\Phi_{2e}(t)$, $e\not=2$. The proof uses
the theory of rational Cherednik algebras (also known as RDAHAs) over similar localizations of ${\mathbb C}[t,t^{-1}]$. In future paper,
the authors expect to apply these results to prove global versions of the conjecture, at least in the equal parameter case
with bad primes excluded. 
 
\end{abstract}

\title[Extending Hecke endomorphism algebras]{Extending Hecke endomorphism algebras}
 \author{Jie Du}
 \address{School of Mathematics and Statistics\\University of New South Wales\\ UNSW Sydney 2052
}
 \email{j.du@unsw.edu.au}
 \author{Brian J. Parshall}
\address{Department of Mathematics \\
University of Virginia\\
Charlottesville, VA 22903} \email{bjp8w@virginia.edu {\text{\rm
(Parshall)}}}
\author{Leonard L. Scott}
\address{Department of Mathematics \\
University of Virginia\\
Charlottesville, VA 22903} \email{lls2l@virginia.edu {\text{\rm
(Scott)}}}
\thanks{Research supported in part by the Australian Research Council and National Science
Foundation}
\dedicatory{We dedicate this paper to the memory of Robert Steinberg.}
\maketitle
\section{Introduction}   Let $\sG=\{G(q)\}$ be a family of finite groups of Lie type having irreducible (finite)
Coxeter system $(W,S)$ \cite[(68.22)]{CR87}. The pair $(W,S)$ remains fixed throughout this paper. Let $B(q)$ be a Borel subgroup of $G(q)$. There are index parameters $c_s\in\mathbb Z$, $s\in S$, defined by
$$[B(q):\,{}^sB(q)\cap B(q)]=q^{c_s}, \quad s\in S.$$

The generic Hecke algebra $\sH$ over the ring $\sZ={\mathbb Z}[t,t^{-1}]$ of Laurent polynomials associated to $\sG$  has basis $T_w$, $w\in W$, subject
to relations
\begin{equation}\label{relations}T_sT_w=\begin{cases} T_{sw},\quad sw>w;\\
t^{2c_s}T_{sw} + (t^{2c_s}-1)T_w,\quad sw<w\end{cases}\end{equation}
for $s\in S,w\in W$.
This algebra is defined just using $t^2$, but it is convenient to have its square root $t$ available.
We call $\sH$ a Hecke algebra of Lie type over $\sZ$. It is related to the representation theory of the groups in $\sG$
as follows: for any prime power $q$, let $R$ be any field (we will shortly allow $R$ to be a ring) in which the
integer $q$ is invertible and has a square root $\sqrt{q}$. Let $\sH_R=\sH\otimes_{\sZ}R$ be the 
  algebra obtained by base change through the
map $\sZ\to R$, $t\mapsto \sqrt{q}$. Then $\sH_R\cong\End_{G(q)}(\ind_{B(q)}^{G(q)}{R})$. Thus, the
 generic Hecke algebra  $\sH$ is the quantumization (in the sense of \cite[\S0.4]{DDPW08}) of an infinite family of important endomorphism algebras.
 
In type $A_n$, i.e., when $\sG=\{GL_{n+1}(q)\}$,  one can also consider the $q$-Schur algebras, viz., algebras Morita equivalent to
\begin{equation} \label{qSchur} S_R:=\End_{\sH_R}\left(\bigoplus_{J\subseteq S}\ind_{\sH_{J,R}}^{\sH_R}\IND_J\right).\end{equation}
In this case, $S_R$ is a quasi-hereditary algebra whose representation theory is closely related to that of the quantum general linear groups. The $q$-Schur algebras have historically played an important 
role in representation theory of the finite general linear groups, thanks to the work of Dipper, James, and others. More
generally, although the definition (\ref{qSchur}) makes sense in all types, less is known about its properties or the precise role
it plays in the representation theory or homological algebra of the corresponding groups in $\sG$. The purpose of this paper, and its sequels,
is to enhance $S_R$ in a way described below, so that it does become relevant to these questions.

\medskip\noindent
\subsection{ Stratifying systems.} At this point, it will be useful to review the notion of a {\it strict stratifying system} for an algebra.  These systems provide a framework for studying algebras similar to quasi-hereditary algebras. They appear in the statement of the first main Theorem \ref{4.4}.  Although the algebras in the Theorem \ref{4.4} are shown later
 to be quasi-hereditary, the theory of stratifying systems is useful both in providing a framework and as a tool
in obtaining the final results.

First, recall 
that a preorder on a set $X$ is a transitive and reflexive relation $\leq$. The associated equivalence relation $\sim$  on $X$ is defined by setting, for $x,y\in X$,
$$x\sim y \iff x \leq y \,\&\, y \leq x.$$
A preorder induces an evident partial order, still denoted $\leq$, on the set $\overline X$ of equivalence classes of $\sim$. In this paper, a set $X$ with a preorder is called a quasi-poset. Also, if $x\in X$, let $\bar x\in \overline X$ be its associated equivalence
class.

Now let $R$ be a Noetherian commutative ring, and let  $A$ be an $R$-algebra, finitely generated and projective as an $R$-module. Let $\Lambda$ be a finite quasi-poset. For
each $\lambda\in\Lambda$, it is required that there is given a finitely generated $A$-module $\Delta(\lambda)$ and a finitely generated projective 
$ A$-module
$P(\lambda)$ together with a fixed surjective $P(\lambda)\twoheadrightarrow \Delta(\lambda)$ morphism of $A$-modules. The following
conditions are required:

\begin{itemize}
\item[(1)] For $\lambda,\mu\in\Lambda$, 
$$\Hom_{ A}(P(\lambda), \Delta(\mu))\not=0\implies \lambda\leq\mu.$$

\item[(2)] Every irreducible $A$-module $L$ is a homomorphic image of some $\Delta(\lambda)$.

\item[(3)] For $\lambda\in\Lambda$, the $A$-module $P(\lambda)$ has a finite filtration by $ A$-submodules
with top section $\Delta(\lambda)$ and other sections of the form $\Delta(\mu)$ with $\bar\mu>\bar\lambda$.
\end{itemize}
When these conditions all hold, the data $\{\Delta(\lambda)\}_{\lambda\in\Lambda}$ is a  {\it strict stratifying system} for
 $ \Amod$. It is also clear that $\Delta(\lambda)_{R'}, P(\lambda)_{R'}, \dots$ is a
strict stratifying system for $ A_{R'}$-mod {\it for any base change $R\to R'$}, provided $R'$ is a Noetherian commutative 
ring. (Notice that condition (2) is redundant, if it is known that the direct sum of the projective modules in (3) is a
progenerator---a property preserved by base change.)

An ideal $J$ in an $R$-algebra $ A$ as above is called a {\it stratifying ideal} provided that $J$ is an $R$-direct summand
of $ A$ (or equivalently, the inclusion $J\hookrightarrow A$ is $R$-split), and for $A/J$-modules $M,N$ inflation defines an isomorphism
\begin{equation}\label{recollement}
\Ext^n_{A/J}(M,N)\overset\sim\longrightarrow\Ext^n_{ A}(M,N),\quad\forall n\geq 0.\end{equation}
of Ext-groups. A {\it standard stratification} of length $n$ of $A$ is a sequence $0=J_0\subsetneq J_1\subsetneq
\cdots\subsetneq J_n= A$ of stratifying ideals of $A$ such that each $J_i/J_{i-1}$ is a projective
$A/J_{i-1}$-module. If $ \Amod$ has a strict stratifying system with quasi-poset $\Lambda$, then it has a standard stratification of length $n=|\bar\Lambda|$; see \cite[Thm. 1.2.8]{DPS98a}. 

 In the case of a finite dimensional algebra $A$ over a field $k$, the notion of a strict stratifying system 
 $\{\Delta(\lambda)\}_{\lambda\in\Lambda}$ for $\Amod$ simplifies
 somewhat. In this case, it can be assumed that each $\Delta(\lambda)$ has an irreducible head $L(\lambda)$,
 that $\lambda\not=\mu\implies L(\lambda)\not\cong L(\mu)$, and that $P(\lambda)$ is indecomposable. Two caveats are
 in order, however: (i) it may be necessary to enlarge the base set $\Lambda$ to be able to index all the irreducible
 modules, though $\overline{\Lambda}$ can remain the same; (ii)  it may be easier to verify
 (1), (2) and (3) over a larger ring and then base change. The irreducible head versions of the $\Delta(\lambda)$'s
 can then be obtained as direct summands of the base-changed versions.

When the algebra $A$  arises as an endomorphism algebra $A=\End_B(T)$, there is a useful theory for obtaining 
a strict stratifying system for $\Amod$. In fact, this is how such stratifying systems initially arose (see
\cite{CPS96} and \cite{DPS98a}). This approach is followed in the proof of the main theorem in this paper. For
convenience, we summarize the sufficient conditions that will be used, all taken from \cite[Thm. 1.2.10]{DPS98a}.

\begin{thm}\label{strattheorem}  Let $B$ be a finitely generated projective $R$-algebra and let $T$ be a finitely generated right
$B$-module which is projective over $R$. Define $A:=\End_B(T)$. Assume that $T=\oplus_{\lambda\in\Lambda}T_\lambda$, where $\Lambda$ is a finite quasi-poset. For $\lambda\in\Lambda$, assume there is given a fixed $R$-submodule
$S_\lambda\subseteq T_\lambda$ and an increasing filtration $F^\bullet_\lambda:0=F_\lambda^0\subseteq F_\lambda^1
\subseteq\cdots\subseteq F_\lambda^{t(\lambda)}=T_\lambda$ satisfying the following conditions:
\begin{itemize}
\item[(1)] For $\lambda\in\Lambda$, $F_\lambda^\bullet$ has bottom section $F_\lambda^1/F_\lambda^0
\cong S_\lambda$, and, higher sections $F_\lambda^{i+1}/F_\lambda^i$ ($1\leq i \leq t(\lambda)-1$) of the form $S_\nu$ with $\bar\nu>\bar\lambda$.
\item[(2)] For $\lambda,\mu\in\Lambda$, $\Hom_B(S_\mu,T_\lambda)\not=0\implies \lambda\leq\mu$.
\item[(3)] For $\lambda\in\Lambda$, $\Ext^1_B(T_\lambda/F^i_\lambda,T)=0$ for all $i$.
\end{itemize}
Let $A=\End_B(T)$, and, for $\lambda\in\Lambda$, define $\Delta(\lambda):=\Hom_B(S_\lambda,T)\in A$-mod. 
Assume that each $\Delta(\lambda)$ is $R$-projective. Then $\{\Delta(\lambda)\}_{\lambda\in\Lambda}$ is a
strict stratifying system for $\Amod$.
\end{thm} 

It is interesting to note that these sufficient conditions are not, in general, preserved under base change, though
the resulting strict stratifying systems are preserved (becoming strict stratifying systems for the base-changed
version of the algebra $A$).

 \medskip
 \subsection{ Cells and {\it q}-permutation modules.} We assume familiarity with Kazhdan-Lusztig cell theory
 for the Coxeter systems $(W,S)$. See, for instance, \cite{DDPW08} and \cite{Lus03}.
In Conjecture \ref{conjecture} below and in the main Theorem \ref{4.4}, the set $\Lambda$ will be the set
$\Omega$ of left Kazhdan-Lusztig cells for $(W,S)$.  For each $\omega\in\Omega$, let 
\begin{equation}\label{defnleftcell}S(\omega):=\sH^{\leq_L\omega}/\sH^{<_L\omega}\in \sH{\text{\rm --mod}}\end{equation}
 be the corresponding left cell module. It is known that $S(\omega)$ is a free $\sZ$-module with basis
 corresponding to certain Kazhdan-Lusztig basis elements $C'_x$, $x\in\omega$; see \S2.
 The corresponding dual left cell module is defined
 \begin{equation}\label{dualleftcell} 
 S_\omega:=\Hom_\sZ(S(\omega),\sZ)\in\, {\text{\rm mod--}}\sH.\end{equation}
 It is regarded as a {\it right} $\sH$-module.  Because $S(\omega)$ and hence $S_\omega$ are free over $\sZ$,
 if $R$ is a commutative $\sZ$-module, we can define
 $$\begin{cases} S_R(\omega):=S(\omega)\otimes_\sZ R \\
 S_{\omega,R}:= S_{\omega}\otimes_\sZ R=\Hom_{R}(S_R(\omega),R).\end{cases}$$
 For the special choice $R=\sQ$---see (\ref{modular system}) below for the definition of $\sQ$---we also use the notations
 \begin{equation}\label{tildenotation}
 \begin{cases}\wS(\omega):=S_\sQ(\omega),\\ \wS_\omega:= S_{\omega,\sQ},\quad \omega\in\Omega.\end{cases}
 \end{equation}

 In addition, for $\lambda\subseteq S$,  let $W_\lambda$ be the parabolic subgroup of $W$ generated by the $s\in \lambda$, and put 
$x_\lambda=\sum_{w\in W_\lambda}T_w$, with $T_w$ as in (\ref{relations}) above. The induced modules $x_\lambda\sH$ (also called $q$-permutation modules) have an increasing filtration with sections
$S_\omega$ for various $\omega\in \Omega$ (precisely, those left cells $\omega$ whose right set ${\sR}(\omega)$ contains $\lambda$). 

Let $\sT=\bigoplus_\lambda x_\lambda\sH$, and $\sA:=\End_{\sH}(\sT)$. For $\omega\in\Omega$,
put $\Delta(\omega):=\Hom_\sH(S_\omega,\sT)\in\sA$-mod. The algebra $\sA$ is very well behaved in type A, a $q$-Schur algebra, and a theme of \cite{DPS98a} was that suitable enlargements, appropriately compatible with two-sided cell theory, should have similar good properties for all types.

Each two-sided cell may be identified with the set of left cells it contains, and the resulting collection $\overline \Omega$ of sets of left cells is a partition of $\Omega$. There are various natural preorders on $\Omega$, but we will be mainly interested in those whose associated equivalence relation has precisely the set $\overline\Omega$ as its associated partition. We call such a preorder {\it strictly compatible} with $\overline \Omega$.

\medskip\noindent
\subsection{A conjecture.} Now we are ready to state the following conjecture, which is a variation (see the Appendix) on \cite[Conj. 2.5.2]{DPS98a}.  We informally think of the algebra
$\sA^+$ in the conjecture as an extension of $\sA$ as a Hecke endomorphism algebra (justifying the title of the paper).

\begin{conj}\label{conjecture} {\it There exists a preorder $\leq$ on the set $\Omega$ of left cells in $W$, strictly compatible with its partition $\overline \Omega$ into two-sided cells, and a right $\sH$-module $\sX$ such that the following statements hold:

\begin{enumerate}
\item[(1)] $\sX$ has an finite filtration with sections of the form $S_\omega$, $\omega\in\Omega$.

\item[(2)] Let $\sT^+:=\!\sT\oplus \sX$ and put 
$$\begin{cases} \sA^+:=\End_{\sH}(\sT^+) ,\\
\Delta^+(\omega):=\Hom_{\sH}(S_\omega,\sT^+),\,\,{\text{\rm for any $\omega\in\Omega$.}}\end{cases}$$
Then, for any commutative, Noetherian\! $\sZ$-algebra\!  $R$, the set $\{\Delta^+(\omega)_R\}_{\omega\in\Omega}$ is a strict stratifying system for $\sA^+_R$-mod relative to the quasi-poset\!
$(\Omega,\leq)$. 
\end{enumerate}}
\end{conj}

The main result of this paper, given in Theorem \ref{4.4}, establishes a special ``local case" of this
conjecture. 
A more detailed description of this theorem requires some preliminary
notation. Throughout this paper, $e$ is positive integer ($\not=2$ in our main results). Let $\Phi_{2e}(t)$
denote the (cyclotomic) minimum polynomial for a primitive $2e$th root of unity $\sqrt{\zeta}=\exp({2\pi i/2e})\in{\mathbb C}$. 
Fix a modular system $(K,\sQ, k)$ by letting
\begin{equation}\label{modular system}
\begin{cases}
\sQ:={\mathbb Q}[t,t^{-1]}]_{\mathfrak p},{\text{\rm where} }\,\,{\mathfrak p}=(\Phi_{2e}(t));\\
K:=  {\mathbb Q}(t),\,\, {\text{\rm the fraction field of }}\,\, \sQ; \\
k:=\sQ/{\mathfrak m}\cong {\mathbb Q}(\sqrt{\zeta}),\,{\text{\rm the residue field of }} \,\,\sQ.
\end{cases}
\end{equation}      

Here $\mathfrak m$ denotes the maximal ideal of the the DVR $\sQ$. With some abuse of notation, we sometimes identify $\sqrt{\zeta}$ with its image in $k$. (Without passing to an extension or completion, the ring $\sQ$ might not have such a
root of unity in it.)
The algebra $\sH_{{\mathbb Q}(t)}$ is split semisimple, with irreducible
modules corresponding to the irreducible modules of the group algebra ${\mathbb Q}W$. The $\sQ$-algebra
\begin{equation}\label{tildeH}
\widetilde\sH:=\sH\otimes_\sZ\sQ \end{equation}
has a presentation by elements $T_w\otimes 1$ (which will still be denoted
$T_w$, $w\in W$) completely analogous to (\ref{relations}). Similar remarks apply to
$\sH_k$,  replacing $t^2$ by $\zeta$.
  Then Theorem \ref{4.4} establishes that there
exists
 a $\widetilde\sH$-module $\widetilde{\mathcal X}$ which is filtered by
dual left cell modules $\wS_\omega$ such that the analogues of conditions (1) and (2) over $\sQ$ in Conjecture \ref{conjecture} hold.  The preorder used in Theorem \ref{4.4} is constructed as in \cite{GGOR03} from a
``sorting function" $f$, and is discussed in detail in the next section.

With more work, it 
can be shown, when $e\not=2$, that the $\sQ$-algebra $\widetilde\sA^+:=\sA^+_\sQ$ is quasi-hereditary. This
is done in Theorem \ref{cases}.  Then Theorem \ref{cor6.5} identifies the module category for a base-changed
version of this algebra with a RDAHA-category $\sO$ in \cite{GGOR03}.  Such an identification in type $A$ was conjectured
 in \cite{GGOR03}, and then proved by Rouquier in \cite{Ro08} (when $e\not=2$).

Generally speaking, this paper focuses on the ``equal parameter" case (i.e., all $c_s=1$ in (1.0.1)), which
covers the Hecke algebras relevant to all untwisted finite Chevalley groups. We will assume this condition unless
explicitly stated otherwise, avoiding  a number of complications involving Kazhdan-Lusztig basis elements and Lusztig's algebra $\sJ$. In this context, the critical Proposition \ref{2.2} depends on results of
\cite{GGOR03} which, in part, were only determined in the equal parameter case.  Nevertheless, much of our discussion applies in the unequal parameter cases. In particular, we mention that the elementary, but important,
Lemma \ref{NM} is stated and proved using unequal parameter notation. This encourages the authors to believe the main
results are also provable in the unequal parameter case, though this has not yet been carried out. Note that all the
rank 2 cases are treated in \cite{DPS98a}, leaving the quasi-split cases with rank $>2$. All these quasi-split cases have parameters confined to the set $\{1,2,3\}$.

\section{Some preliminaries}
This section recalls some mostly well-known facts and fixes notation regarding cell theory.
Let $W$ be a finite Weyl group associated to a finite root system $\Phi$ with a fixed set of simple roots $\Pi$.  Let $S:=\{s_\alpha\,|\,\alpha\in\Pi\}$.  Let $\sH$ is a Hecke algebra over $\sZ$ defined by (\ref{relations}). We assume (unless explicitly noted
otherwise) that each $c_s=1$ for $s\in S$. Thus, $(W,S)$ corresponds, in the language of the introduction, to some types of split Chevalley groups, though we will have no further need of that context.
 Let 
$$C'_w=t^{-l(w)}\sum_{y\leq w}P_{y,w}T_y,$$
where the $P_{y,w}$ is a Kazhdan-Lusztig polynomial in ${\mathfrak q}:=t^2$. 
Then $\{C'_w\}_{w\in W}$ is a Kazhdan--Lusztig (or canonical) basis for $\sH$. The element $C'_x$ is denoted $c_x$ 
in \cite{Lus03}, a reference we frequently quote. Let $h_{x,y,z}\in\sZ$ denote the structure constants. In other words, 
$$C'_xC'_y=\sum_{z\in W}h_{x,y,z}C'_z.$$
Using the preorders $\leq_L$ and $\leq_R$ on $W$, the positivity (see \cite[\S7.8]{DDPW08}) of the coefficients of the $h_{x,y,z}$ implies 
\begin{equation}\label{hxyz}
h_{x,y,z}\neq0\implies  z\leq_Ly, z\leq_R x
\end{equation}

The Lusztig function $a:W\longrightarrow{\mathbb N}$ is defined as follows. For $z\in W$, let $a(z)$ be the smallest
 nonnegative integer such that $t^{a(z)}h_{x,y,z}\in
{\mathbb N}[t]$ for all $x,y\in W$. It may equally be defined as the smallest nonnegative integer such that
$t^{-a(x)}h_{x,y,z}\in{\mathbb N}[t^{-1}]$, as used in \cite{Lus03} (or see \cite[\S7.8]{DDPW08}).  In fact, each $h_{x,y,z}$ is invariant under the automorphism
$\sZ\to \sZ$ sending $t$ to $t^{-1}$.
It is not difficult to see that $a(z)=a(z^{-1})$. For $x,y,z\in W$,
let $\gamma_{x,y,z}$ be the coefficient of $t^{-a(z)}$ in $h_{x,y,z^{-1}}$. 
Also,  by
\cite[Conjs. 14.2({\bf P8}),15.6]{Lus03},
\begin{equation}\label{gamma}
\gamma_{x,y,z}\neq0\implies x\sim_Ly^{-1}, y\sim_Lz^{-1},z\sim_Lx^{-1}.
\end{equation}

The function $a$ is constant on two-sided cells in $W$, and so can be regarded as a function (with values in
${\mathbb N}$ on (a) the set of two-sided cells; (b) the set of left (or right) cells; and (c) the set $\irr({\mathbb Q}W)$
of irreducible ${\mathbb Q}W$-modules.\footnote{It is well-known that $\mathbb Q$ is a splitting field
for $W$ \cite{B71}.} In addition, $a$ is related to the generic degrees $d_E$, $E\in\irr({\mathbb Q}W)$.
For $E\in\irr({\mathbb Q}W)$, let
$d_E= bt^{a_E} + \cdots + ct^{A_E},$ with $a_E\leq A_E$ and $bc\not=0$, so that $t^{a_E}$ (resp., $t^{A_E}$) is the lowest (resp., largest) power
of $t$ appearing nontrivially in $d_E$. Then $a_E=a(E)$; cf. \cite[Prop. 20.6]{Lus03}. Also, as noted in \cite[\S6]{GGOR03},  $A_E=N-a(E\otimes\det)$, where $N$ is the number
of positive roots in $\Phi$.
Following \cite[\S6]{GGOR03}, we will use the ``sorting function" $f:\irr({\mathbb Q}W)\to\mathbb N$ defined by
\begin{equation}\label{sorting} f(E)= a_E +A_E=a(E)+N-a(E\otimes\det).\end{equation}
The function $f$ is also constant on two-sided cells: if $E$ is an irreducible ${\mathbb Q}W$-module associated a two-sided cell $\bold c$, 
then $E\otimes \det$ is an irreducible module associated to the two-sided cell $w_0{\bold c}$. See \cite[Lem. 5.14(iii)]{Lus84}

The function $f$ is used in \cite{GGOR03} to define various order structures on the set $\irr({\mathbb Q}W)$ of irreducible ${\mathbb Q}W$-modules. Put
$E<_fE'$ (our notation) provided $f(E)<f(E')$. There are at least two natural ways to extend $<_f$ to a preorder. 
The first way, which is only in the background for us, is to set $E\preceq_f E'\iff E\cong E'$ or $E<_fE'$. This gives a poset
structure, and is used, in effect, by \cite{GGOR03} for defining a highest category $\sO$; see \cite[\S2.5, 6.2.1]{GGOR03}.

We use $<_f$ here to define a preorder $\leq_f$ on the set $\Omega$ of left cells:  First, observe that the function
$f$ above is constant on irreducible modules associated to the same left cell (or even the same two-sided cell) 
and
so may be viewed as a function on $\Omega$. 
We can now define the (somewhat subtle) preorder $\leq_f$ on
$\Omega$ by setting $\omega\leq_f\omega'$ (for $\omega,\omega'\in\Omega$) if and only if either $f(\omega)<f(\omega')$, or $\omega$ and $\omega'$
{\it lie in the same two-sided cell.}  Note that the ``equivalence classes" of the preorder $\leq_f$  identify with the set of two-sided cells---thus, $\leq_f$
is strictly compatible with the set of two-sided cells in the sense of \S1.
Also,
\begin{equation}\label{LRorder}
E<_{LR}E' \implies E'<_f E;
\end{equation}
 see \cite[Lem. 6.6]{GGOR03}. Here $E,E'$ are in $\irr({\mathbb Q}W)$, and the notation $E<_{LR}E'$ means
 that the two-sided cell associated with $E$ is strictly smaller than that associated with $E'$, with respect to
 the Kazhdan-Lusztig order on two-sided cells. A  property similar to (\ref{LRorder}) holds if $<_{LR}$ is
 replaced with $<_L$, defined similarly, but using left cells. In terms of $\Omega$, this left cell version reads:
\begin{equation}\label{LRorder2}
\omega,\omega'\in\Omega, \omega<_L\omega'\implies f(\omega)>f(\omega').\end{equation}
Notice that (\ref{LRorder2}) follows from (\ref{LRorder}) using \cite[Cor. 1.9(c)]{Lus87b}. (The latter result
implies that $\omega,\omega'$ on the left in (\ref{LRorder2}) cannot belong to the same two-sided cell.)
Thus, the preorder $\leq_f$ is a refinement of the preorder  $\leq_{L}^{\text{\rm op}}$ on $\Omega$, and $\leq_f$ induces on the set of two-sided cells a refinement of the partial order
$\leq_{LR}^{\text{\rm op}}$. For further discussion, see the Appendix.

\section{(Dual) Specht modules of Ginzburg--Guay--Opdam--Rouquier}

The asymptotic form $\sJ$ of $\sH$ is a ring with $\mathbb Z$-basis $\{j_x\mid x\in W\}$ and multiplication
$$j_xj_y=\sum_{z}\gamma_{x,y,z^{-1}}j_z.$$ 
This ring was originally introduced in \cite{Lus87b}, though we follow \cite[18.3]{Lus03},  using
 a slightly different notation. 
 
 \subsection{The mapping $\varpi$ and its properties.} As per
\cite[18.9]{Lus03}, define
 a $\sZ$-algebra homomorphism
 \begin{equation}\label{themap}\varpi:\sH\to \sJ_{\sZ}=\sJ\otimes\sZ,\quad C'_w\longmapsto \sum_{z\in W}\sum_{d\in\mathbf D\atop a(d)=a(z)}h_{w,d,z}j_z,\end{equation}
 where $\mathbf D$ is the set
 of distinguished involutions in $W$.  Also, for any $\sZ$-algebra $R$, there is also an algebra homomorphism $\varpi_R:\sH_R=\sH\otimes_\sZ R
 \to \sJ_R=\sJ_\sZ\otimes_\sZ R$, obtained by base change. In obvious cases, we often drop the subscript $R$ from $\varpi_R$. 
 
 In particular, $\varpi_{{\mathbb Q}(t)}$ becomes an isomorphism 
 \begin{equation} \label{maptwo}\varpi=\varpi_{{\mathbb Q}(t)}:\sH_{\mathbb Q(t)}\overset\sim\longrightarrow\sJ_{\mathbb Q(t)}.\end{equation}
 See \cite{Lus87b}.
 Also, $\varpi$ induces a monomorphism 
 \begin{equation}\label{varpi}
 \varpi=\varpi_{{\mathbb Q}[t,t^{-1}]}:\sH_{{\mathbb Q}[t,t^{-1}]}\hookrightarrow \sJ_{{\mathbb Q}[t,t^{-1}]}=\sJ_{\mathbb Q}\otimes{\mathbb Q}[t,t^{-1}].
 \end{equation}
 Moreover, base change to ${\mathbb Q[t,t^{-1}]}/(t-1)$ induces an isomorphism 
 \begin{equation}\label{map3}\mathbb \varpi=\varpi_{\mathbb Q}:{\mathbb Q}W\overset\sim\longrightarrow\sJ_{\mathbb Q}\end{equation}
  (cf., \cite[Prop.~1.7]{Lus87c}). This allows us to 
 identify irreducible $\mathbb QW$-modules with irreducible $\sJ_{\mathbb Q}$-modules.\footnote{The map $\varpi$ is the composition $\phi\circ\dagger$, where $\phi$ and $\dagger$ are defined in \cite[18.9]{Lus03}  and \cite[3.5]{Lus87c}, respectively.  The numbers $\widehat
 n_z$ appearing there (which are $\pm 1$ by definition in \cite[\S18.8]{Lus03}) are all equal to 1, because of the positivity (see \cite[\S7.8]{Lus03}) of the structure constants appearing in \cite[14.1]{Lus03}. This $\varpi$ is not the same one as defined in \cite[p.647]{GGOR03}, where the $C$-basis was used. Nevertheless, the arguments of \cite[\S6]{GGOR03} go through, using the $C'$-basis and our $\varpi$ (see Remark \ref{AfterLemma5.1} below), so
\cite[Thm. 6.8]{GGOR03} guarantees the modules $S_{\mathbb C}(E)$ defined below using our set-up are the same,
at least up to a (two-sided cell preserving) permutation of the isomorphism types labeled by the $E$'s, as the modules
$S(E)$ defined in \cite[Defn. 6.1]{GGOR03} with $R={\mathbb C}$.  The proof of \cite[Thm. 6.8]{GGOR03} also establishes
such an identification of the various modules $S_R(E)$ when $R$ is 
 a completion of ${\mathbb C}[t,t^{-1}]$.}

 For the irreducible (left) $\sJ_{\mathbb Q}$-module identified with $E\in{\text{\rm Irr}}(\mathbb QW)$, the (left) $\sH_{\mathbb Q[t,t^{-1}]}$-module 
 $$S(E):=\varpi^*(E\otimes \mathbb Q[t,t^{-1}])=\varpi^*(E_{\mathbb Q[t,t^{-1}]})$$
 is called here a {\it dual Specht module} for $\sH_{{\mathbb Q}[t,t^{-1}]}$; cf. \cite[Cor.~6.10]{GGOR03}.\footnote{
  In \cite[Defn. 6.1]{GGOR03}, the module $S(E)$ there is called a standard module. Our choice of terminology is justified by the discussion
  following the proof of Lemma~\ref{Lemma5.1} below.}  Note that $S(E)\cong E_{{\mathbb Q}[t,t^{-1}]}$ as a ${\mathbb Q}[t,t^{-1}]$-module. Therefore, $S(E)$ is a free ${\mathbb Q}[t,t^{-1}]$-module.
  Putting $S_E=\Hom_{{\mathbb Q}[t,t^{-1}]}(S(E),{\mathbb Q}[t,t^{-1}])$,
  define
  \begin{equation}\label{Spechtdual}
  \begin{cases} \wS(E):=S_\sQ(E), \\
  \wS_E:=S_{E,\sQ},\end{cases}\end{equation}
  where, in general, for base change to a commutative, Noetherian ${\mathbb Q}[t,t^{-1}]$-algebra $R$,
  $$\begin{cases} 
  S_R(E):=S(E)\otimes_{{\mathbb Q}[t,t^{-1}]} R,\\ S_{E,R}:= S_E\otimes_{{\mathbb Q}]t,t^{-1}]} R\cong
  \Hom_{R}(S_R(E),R).\end{cases}
  $$
  


The following proposition is proved using RDAHAs, and it is the only ingredient in the proof of Theorem \ref{4.4} where
these algebras are used.

\begin{prop}\label{2.2}   Assume that $e\not=2$. Suppose
$E, E'$ are irreducible ${\mathbb Q}W$-modules.  If $E\not\cong E'$ and 
$$\Hom_{\sH_k}(S_k(E), S_k(E'))\not=0,$$ then $f(E)<f({E'})$.   Also,
$\Hom_{\sH_k}(S_k(E),S_k(E))\cong k$.\end{prop}

\begin{proof}Without loss, we replace $k$ in the statement of the proposition by ${\mathbb C}$, using the analogous definitions of $S_{\mathbb C}(E)$. In addition, the statement of the proposition is invariant under any two-sided cell
preserving permutation of the labeling of the irreducible modules. After applying such a permutation on the right (say)
we may assume, by \cite[Thm. 6.8]{GGOR03}  and taking into account ftn.~1, that
$${\text{\rm KZ}}(\Delta(E))\cong S_{\mathbb C}(E),$$
where 
\begin{enumerate}
\item[(1)] $\Delta(E)$ is the standard module for a highest weight category $\sO$ given in \cite{GGOR03},
having partial order $\leq_f$ (see \cite[Lem.~2.9, 6.2.1]{GGOR03}) on its set of isomorphism classes of irreducible modules, which are indexed by
isomorphism classes of irreducible $ {\mathbb Q}W$-modules . We take $k_{H,1}= 1/e>0$ in \cite{GGOR03} above Thm. 6.8
 and in Rem. 3.2 there.
 
 \item[(2)] The functor $ {\text{\rm KZ}}:\sO\longrightarrow \overline\sO$ is naturally isomorphic to 
 the quotient  map $M\mapsto\overline M$ in \cite[Prop. 5.9, Thm.~5.14]{GGOR03}, the quotient category there identifying with $\sH_{\mathbb C}$-mod.
 
 \end{enumerate}
 
 Using \cite[Prop. 5.9]{GGOR03}, which requires $e\not=2$, we have, for any irreducible ${\mathbb C}W$-modules $E,E'$, 
 $$\Hom_\sO(\Delta(E),\Delta(E'))\cong\Hom_{\overline\sO}(\overline\Delta(E),\overline\Delta(E'))
 \cong\Hom_{\sH_{\mathbb C}}(S_{\mathbb C}(E), S_{\mathbb C}(E')).$$
 If $E\not\cong E'$, then $\Delta(E)\not\cong\Delta(E'))$ and $\Hom_\sO(\Delta(E),\Delta(E'))\not=0$
 implies that $E<_fE'$, i.e., $f(E)<f(E')$.
 
 On the other hand, if $E\cong E'$, then $\Hom_\sO(\Delta(E),\Delta(E'))\cong\mathbb C$. This implies
 $$\Hom_{\sH_{\mathbb C}}(S_{\mathbb C}(E),S_{\mathbb C}(E'))\cong {\mathbb C}.$$ 
 Returning to the original $k={\mathbb Q}(\sqrt{\zeta})$, we may
 conclude the same isomorphism holds in the original setting as well.
\end{proof}

\begin{cor}\label{ext1} Assume $e\not=2$. Let $E,E'$ be irreducible ${\mathbb Q}W$-modules. Then 
$$\Ext^1_{\wH}(\wS(E), \wS(E'))\not=0\implies f(E)<f(E').$$
In particular, $\Ext^1_{\wH}(\wS(E), \wS(E))=0$.
\end{cor}
\begin{proof} In (\ref{modular system}) let $\pi=\Phi_{2e}(t)$ be the generator of the maximal ideal $\mathfrak m$
of $\sQ$, and consider the short exact sequence
$$0\to \wS(E')\overset\pi\longrightarrow \wS(E')\to S_k(E')\to 0.$$  By the long exact sequence of $\Ext$, there is  an exact sequence 
$$\begin{aligned} 0\to \Hom_\wH(\wS(E),\wS(E'))&\overset\pi\to \Hom_\wH(\wS(E),\wS(E'))\to \Hom_{\wH_k}(S_k(E),S_k(E'))\\ &\to 
\Ext^1_\wH(\wS(E),\wS(E'))\overset\pi\to\Ext^1_\wH(\wS(E),\wS(E'))\\ &\to \Ext^1_{\wH_k}(S_k(E),S_k(E')).\end{aligned}$$
Because $\sH_{{\mathbb Q}(t)}=\wH_{{\mathbb Q}(t)}$ is semisimple, 
$$\Ext^1_\wH(\wS(E),\wS(E'))_{{\mathbb Q(t)}}\cong\Ext^1_{\sH_{{\mathbb Q}(t)}}(S(E)_{{\mathbb Q}(t)}, S(E')_{{\mathbb Q}(t)})=0.$$
In other words, if it is nonzero, $\Ext^1_{\wH}(\wS(E),\wS(E'))$ is a torsion module, so that
the map $\Ext^1_\wH(\wS(E),\wS(E'))\overset\pi\to\Ext^1_\wH(\wS(E),\wS(E'))$ is not injective. Thus, it suffices
 to prove that when
  $f(E)\not<f(E')$, the map 
  \begin{equation}\label{surjective}
  \Hom_{\wH}(\wS(E),\wS(E'))\to
\Hom_{\wH_k}(S_k(E),S_k(E'))\end{equation}
is surjective. If $E\not\cong E'$, Proposition \ref{2.2} gives $\Hom_{\wH_k}(S_k(E), S_k(E'))=0$ implying the surjectivity of (\ref{surjective}) trivially. On the other hand, if $E\cong E'$, the proposition  gives
$\Hom_{\sH_k}(S_k(E),S_k(E'))\cong k$.  This also gives surjectivity of the map in (\ref{surjective}), since it becomes surjective 
upon restriction to $\sQ\subseteq \Hom_{\wH}(\wS(E),\wS(E'))$ (taking $E'=E$). \end{proof}

 \section{Two preliminary lemmas} 
Let $R$ be a commutative ring and let $\sC$ be an abelian $R$-category.  For $A,B\in\sC$, let
$\Ext^1_\sC(A,B)$ denote the Yoneda groups of extensions of $A$ by $B$. (We   do not
require the higher Ext-groups in this section.) Let $M,Y\in\sC$, and suppose that
$\Ext^1_\sC(M, Y)$ is generated as an $R$-module by elements $\epsilon_1,\cdots, \epsilon_m$.  Let $\chi:=\oplus_i\epsilon_i\in\Ext^1_\sC(
M^{\oplus m}, Y)$ correspond to the short exact sequence $0\to Y\to  X\to M^{\oplus m}\to 0$.

\begin{lem} \label{abstract1}
The map $\Ext^1_\sC(M, Y)\to\Ext^1_\sC(M, X)$, induced by the inclusion $ Y\to X$, is
the zero map.\end{lem}

\begin{proof} Using the ``long" exact sequence of $\Ext^\bullet$, it suffices to show that the map
$\delta$ in the sequence
$$\Hom_\sC(M, X)\longrightarrow\Hom_\sC(M,M^{\oplus m})\overset\delta\longrightarrow\Ext^1_\sC(M, Y)$$
is surjective---equivalently, that each $\epsilon_i\in\Ext^1_\sC(M, Y)$ lies in the image of $\delta$. Let 
$0\to Y\to X_i\to M\to 0$ correspond to $\epsilon_i\in\Ext^1_\sC(M, Y)$. By construction, $\epsilon_i$
is the image of $\chi$ under the natural map 
$$j^*_i:\Ext^1_\sC(M^{\oplus m}, Y)\to\Ext^1_\sC(M, Y),$$ which is the pull-back of the inclusion
$j_i$ of $M$ into the $i$th summand of $M^{\oplus m}$. So there is a natural commutative diagram
$$\begin{CD}
0 @>>> Y @>>>  X @>>>M^{\oplus m} @>>> 0 \\ @. @AAA  @AAA @AA{j_i}A @. \\
0 @>>>  Y @>>>  X_i @>>> M @>>> 0
\end{CD}.$$
 There is a corresponding
commutative diagram
\begin{equation}\label{commdiagram}
\begin{CD} \Hom_\sC(M, X) @>>> \Hom_\sC(M,M^{\oplus m}) @>^\delta>>\Ext^1_\sC(M, Y) \\
@AAA @AAA @| \\
\Hom_\sC(M, X_i) @>>> \Hom_\sC(M,M)  @>^{\delta_i}>> \Ext^1_\sC(M, Y)  \end{CD}\end{equation}
where each row is part of a ``long" exact sequence. Then $\delta_i(1_{M})=\epsilon_i$.
Therefore, the commutativity of the right-hand square in (\ref{commdiagram}) immediately says that
$\epsilon_i$ lies in the image of $\delta$. \end{proof}


This lemma together with the additivity of the functor $\Ext^1_\sC$ gives immediately the following.
\begin{cor}\label{abstract2}
Maintain the set-up above. If $\Ext^1_\sC(M,M)=0$, then $\Ext^1_\sC(M, X)=0$.
\end{cor}

Next, let $R$ be a commutative ring which is a $\sZ$-algebra and write $q=t^2\cdot1$, the image in $R$ of $t^2\in\sZ$.
For the rest of this section, we allow general parameters $c_s$, $s\in S$, in (\ref{relations}).

\begin{lem} \label{NM}
Let $\fN\subseteq\fM$ be left ideals in $\sH_R$, with each spanned by the Kazhdan--Lusztig basis elements $C'_y$ that they contain.
Let $s\in S$ be a simple reflection and assume either $\fN=0$ or that $q^{c_s}+1$ is not a zero divisor  in $R$. Suppose $0\not=x\in \fM/\fN$ satisfies 
\begin{equation}\label{*_s}
T_s\cdot x=q^{c_s}x.
\end{equation}
Then $x$ is represented in $\fM$ by an $R$-linear combination of Kazhdan--Lusztig basis elements $C'_y$ with $sy<y$.
\end{lem}
\begin{proof} Let $[m]$ denote the image in $\fM/\fN$ of $m\in\fM$. Note that $\fM,\fN$ and $\fM/\fN$ are all $R$-free, since the $C'_y$ which belong to $\fM$ (resp., $\fN$) form a basis for $\fM$ (resp., $\fN$). The $R$-module $\fM/\fN$ has a basis consisting of all $[C'_y]\neq0$ with $C'_y\in\fM$.

Write $x=\sum_{y}a_y[C'_y]$ with $a_y[C'_y]\neq0$ and $C'_y\in\fM$. Observe that, for $y\in W, s\in S$,
\begin{equation}\label{goesdown} sy<y\implies T_sC'_y=q^{c_s}C'_y.\end{equation}
 Therefore, in the above expression for $x$, it may also be assumed that  $sy>y$ for each nonzero term $a_y[C'_y]$. Let $a_w[C'_w]\neq0$ be chosen with $w$ maximal among these $y$. In general, for $sy>y$, we have
$$T_s C'_y=-C'_y+C'_{sy}+\sum_{z<y\atop sz<z}b_zC'_z$$
for various $b_z\in R$. 
Equating coefficients of $[C'_w]$ gives by (\ref{*_s}) that $(q^{c_s}+1)a_w=0$, since $C'_w$ does not appear with any coefficient in the expressions $T_sC'_y$ with $y\neq w$ and $sy>y$. Now  the hypothesis on zero divisors forces $a_w=0$, a contradiction.
\end{proof}

\begin{rem}\label{manycases} As observed in (\ref{goesdown}) above, elements $x\in\fM/\fN$ satisfying the conclusion
of Lemma \ref{NM} also satisfy its hypothesis (\ref{*_s}). Next, suppose that 
$\la\subseteq S$ and $\fL$ is any $\sH_R$-module. By Frobenius reciprocity, the $R$-module $\Hom_{\sH_R}(\sH_Rx_\la,\fL)$ identifies with the $R$-submodule $\sX\subseteq\fL$ consisting
 of all $x\in \fL$ satisfying \eqref{*_s}
$\forall s\in\la$. Suppose $\fL$ can
be realized as $\fL=\fM/\fN$, with $\fM,\fN$ as in the statement of Lemma \ref{NM}. If $q^{c_s}+1$ is invertible
in $R$ for all $s\in\lambda$, then the lemma implies that $\sX$ has an $R$-basis consisting of 
all nonzero $[C'_y]$ in $\fL$ with
$sy<y$ for all $s\in\lambda$. 

Thus, if $R'$ is an $R$-algebra, then the $R'$-module $\Hom_{\sH_{R'}}(\sH_{R'}x_\lambda,\fL_{R'})$ has essentially
the ``same basis."    This fact will be used in proving the following corollary.\end{rem}

In the result below, we allow $c_s\not=1$. In case $c_s=1$, assumption (2) is satisfied for $R=\sQ$
if and only if $e\not=2$.

\begin{cor} \label{vanishing}
Suppose $R$ is a commutative domain with fraction field $F$, and assume that $R$ is also a $\sZ$-algebra.
 Let $\la\subseteq S$. Assume that
\begin{itemize}
\item[(1)] $\sH_F$ is semisimple;
\item[(2)] $q^{c_s}+1$ is invertible in $R$, for each $s\in\la$.
\end{itemize} Then, for any dual left cell module $S_{\omega,R}$ over $R$,
$$\Ext^1_{\sH_R}(S_{\omega,R}, x_\la\sH_R)=0.$$
\end{cor}
\begin{proof} Put $\sS:=S_{\omega,R}$. Using condition (1) and  \cite[Lem. (1.2.13)]{DPS98a}, it  suffices to prove, for each $R'
=R/\langle d\rangle$ ($d\in R$), that the map
$$\Hom_{\sH_R}(\sS,x_\la\sH_R)\longrightarrow\Hom_{\sH_{R'}}(\sS_{R'}, x_\la \sH_{R'})$$
is surjective. Here $\sS_{R'}=\sS\otimes_{R}R'$.

By \cite[Lem.~2.1.9]{DPS98a}, the left $\sH_R$-module $(x_\la\sH_R)^*:=\Hom_R(x_\la\sH_R,R)$ is naturally isomorphic to $\sH_R x_\la$. By hypothesis, $\sS=\sL^*$ is the dual of a left cell module $\sL$, $R$-free by definition. Thus, $\sL\cong\sS^*$; also, $(\sH_R x_\lambda)^*\cong x_\la\sH_R$.
There are similar isomorphisms for analogous $R'$-modules (for which we use the same notation $(-)^*$). The functor $(-)^*$ provides a contravariant equivalence from the category of finitely generated $R$-free left $\sH_R$-modules and the corresponding right $\sH_R$-module category. A similar statement holds with $R$ replaced by $R'$. Finally, there is a natural isomorphism $(-)^*\otimes_R R'\overset\sim\to (-\otimes_R R')^*$.

Consequently, it is sufficient to prove that 
$$\Hom_{\sH_R}(\sH_R x_\la,\sL)\longrightarrow\Hom_{\sH_{R'}}(\sH_{R'} x_\la,\sL_{R'})$$
is surjective. (Here $\sL_{R'}$ denotes the left cell module in $\sH_{R'}$ defined by the same left cell 
as $\sL$ for $\sH$.) However, viewing $\sL$ and $\sL_{R'}$ as cell modules (over $\sH_R$ and $\sH_{R'}$, respectively), 
hypothesis (2), Lemma~\ref{NM}, and Remark \ref{manycases} give the ``same basis'' (over $R$ and $R'$, respectively).
\end{proof}

\section{The construction of $\wX_\omega$ and the main theorem}

In this section, we prove the main result of the paper (Theorem \ref{4.4}). 

Let $\sQ$ be as in (\ref{modular system}). Recall that $\wsH$ denotes the
$\sQ$-algebra ${\mathcal H}\otimes_\sZ \sQ$.  In general, modules for $\wsH$ are decorated with a ``tilde" (e.g.,
$\wX$). In particular, we recall from (\ref{tildenotation}) the notations $\wS(\omega)$ and $\wS_\omega$.

\subsection{Preliminaries}

Consider a left cell $\omega$ and let $\sJ_\omega=\sum_{y\in\omega}{\mathbb Z}j_y$. Then (\ref{gamma}) implies that $\sJ_\omega$ is a left $\sJ$-module. Using the momomorphism $\varpi$ in \S3, form the left $\sH$-module
$\varpi^*(\sJ_\omega\otimes\sZ)$, the restriction of the $\sJ_\sZ$-module $\sJ_\omega\otimes\sZ$ to $\sH$.

\begin{lem}\label{Lemma5.1} There is an $\sH$-module isomorphism 
$$\sigma: \varpi^*(\sJ_\omega\otimes\sZ)\longrightarrow S(\omega):=\sH^{\leq_L\omega}/\sH^{<_L\omega}$$ induced by the map $\sigma: \sJ_\sZ\to \sH,\;j_y\mapsto C'_y$. In particular, $\wS(\omega)$ is a direct sum of modules $\wS(E)$ for some $E\in\text{\rm Irr}(\mathbb QW)$.
\end{lem}

\begin{proof} This is a refinement of \cite[18.10]{Lus03}. We first observe that the map $\sigma$ clearly induces a $\sZ$-module isomorphism.
  It remains to check  for $y\in\omega$ that
$$\sigma(\varpi(C'_x)j_y)\equiv C'_xC'_y\mod \sH^{<_L\omega}, \quad (x\in W)$$
The proof of \cite[18.10(a)]{Lus03}\footnote{The main ingredient is \cite[18.9(b)]{Lus03}. As previously noted, the
numbers $\widehat n_z$ may be set equal to 1.}
gives the left-hand equality in the expression
\begin{equation}\label{expression}
\sigma(\varpi(C'_x)j_y)=\sigma(\sum_{u\atop a(y)=a(u)}h_{x,y,u}j_u)=\sum_{u\atop a(y)=a(u)}h_{x,y,u}C'_u\equiv C'_xC'_y\mod \sH^{<_L\omega}.\end{equation}
The middle equality is just the definition of $\sigma$. Finally, the right-hand congruence follows from the fact that,
when $h_{x,y,u}C'_u$ is nonzero mod $\sH^{<_L \omega}$, $u$ must belong to the same left cell $\omega$
as $y$, and hence have the same $a$-value.
\end{proof}

If $W$ is of type $A$ and $\omega$ is the left cell containing the longest word $w_{0,\lambda}$ for a partition $\la$. Then $\varpi^*(\sJ_\omega\otimes\sZ)$ is isomorphic to the left cell module whose dual is the Specht module $S_\lambda$. So $\wS(E)$ above could be called a  ``dual Specht module," with $\wS(E)^*$ a ``Specht module."
The modules $\wS_\omega$ are also candidates for the name ``Specht module'' \cite[p.198]{DPS98a}.

\begin{rem}\label{AfterLemma5.1} A completely analogous result to Lemma \ref{Lemma5.1} holds if the Kazhdan-Lusztig
$C$-basis (instead of the $C'$-basis here) is used, as in \cite{GGOR03}. First, it follows from \cite[(3.2)]{Lus87a} that the
map (which we call $\tau$) $\sZ\to \sZ$, sending $t\mapsto -t$, takes the coefficients $h_{x,y,z}$ to analogous coefficients
for the $C$-basis.  Extend $\tau$ to an automorphism, still denoted $\tau$, of $\sJ_{\sZ}$, taking $j_x$ to its $C$-analogue;
we may put $\tau(j_x)=(-1)^{\ell(x)}j_x$
  Thus, any expression
$h_{x,y,z}j_z$ is sent to a $C$-basis analogue. In particular, $\varpi(C'_x)$ is sent to $\varpi (C_x)$, where the latter
$\varpi$ is taken in the $C$-basis set-up. Now it is clear from (\ref{expression}) that the analogue of Lemma 5.1 holds
in the $C$-basis set-up. Note the resulting left cell modules in $\sH$ do not depend on which canonical basis is used.
This allows an identification of the module $S(\omega)$ in Lemma \ref{Lemma5.1} with its $C$-basis counterpart. 

An analogous result holds for two-sided cells, e.g., the $\sH$-module $\varpi^*(\sJ_{{\underline c}}\otimes_{\mathbb Z}{\mathcal Z})$ in \cite[Cor. 6.4]{GGOR03} does not depend on the whether the $C'$-basis
is used (as in this paper) or the $C$-basis is used (as in \cite{GGOR03}). We do not know, however, if the base-change of
the automorphism $\tau$ to $\sJ_{{\mathbb Q}(t)}$ preserves the isomorphism types of irreducible $\sJ_{{\mathbb Q}(t)}$-modules, though
their associated two-sided cells are preserved. This leads to the ``permutation" language used in ftn. 1. In particular,
we do not know if the bijection noted below \cite[Defn. 6.1]{GGOR03} depends on the choice of $C$- or $C'$-basis set-up,
and could result in one choice leading to an identification which is a (two-sided cell preserving) permutation of the 
other.
\end{rem}

Then we have the following result.

\begin{cor}\label{Cor5.3} Assume that $e\not=2$. For left cells $\omega,\omega'$, we have
 $$\Ext^1_{\wH}(\wS_{\omega}, \wS_{\omega'})\neq0\implies f(\omega)>f(\omega').$$
 \end{cor}
\begin{proof} By Lemma \ref{Lemma5.1} and Corollary~\ref{ext1} (which requires $e\not=2$), $\Ext^1_\wH(\wS(\omega'), \wS(\omega))\neq0$ implies $f(\omega)>f(\omega')$.
\end{proof}

For $\lambda\subseteq S$, the induced (right) $\sH$-module $ x_\lambda\sH$ (see \S1.2) has an increasing filtration
\begin{equation}\label{filtrationF}
F_\lambda^\bullet: 0= F_\lambda^0\subseteq F_\lambda^1\subseteq\cdots\subseteq F_\lambda^{m_\lambda}\end{equation}
with sections $F_\lambda^{i+1}/F_\lambda^i\cong S_{\omega_i}$. The bottom section $F_\lambda^1=
F_\lambda^1/F_\lambda^0\cong S_{\omega_1}$, where $\omega_1$ is the left cell containing the longest word
$w_{\lambda,0}$ in the parabolic subgroup $W_\lambda$. If $i>1$, then $\omega_1>_L\omega_i$. See \cite[(2.3.7)]{DPS98a}. The indexing $\omega_i$ of (some of) the left cells depends on $\lambda$, and is formally ``opposite"
(in reverse order) to that used in \cite{DPS98a}. We write $\omega_\lambda:=\omega_1$ to denote its dependence
of the latter cell on $\lambda$.

\begin{lem}\label{strict}  In the filtration (\ref{filtrationF}), if $i>1$, then  $f(\omega_i)>f(\omega_\la)$.
\end{lem}

\begin{proof} This follows from (\ref{LRorder2}), since $\omega_1>_L\omega_i$  for all $2\leq i\leq m_\lambda$ (as noted above). 
\end{proof}

\noindent 
\subsection{First construction of a module $\widetilde X_\omega.$}\label{firstconstruction}
Let $\omega\in\Omega$ be a fixed left cell. The construction of $\widetilde X_\omega$ relies on Corollary
\ref{Cor5.3}.

We iteratively construct an $\wH$-module $\wX_\omega$, filtered by dual left cell modules, such
 that $\wS_{\omega}\subseteq \wX_\omega$ is the lowest nonzero filtration term, and
$$\Ext^1_{\wsH}(\wS_{\omega'}, \wX_\omega)=0\text{ for all left cells }\omega'.$$ 
It will also be a consequence of the construction that every other filtration term $\wS_\nu$, $\nu\in\Omega$, satisfies
$f(\nu)>f(\omega)$.

For $j\in\mathbb N$,  let 
$$\Omega_j=\{\nu\in\Omega\mid f(\nu)=j\}.$$
Fix $i=f(\omega)$. Suppose $\Ext^1_{\wsH}(\wS_\tau,\wS_\omega)\neq0$ for some $\tau\in\Omega$. Then, by the Corollary \ref{Cor5.3}, $f(\tau)>f(\omega)=i$. Assume $f(\tau)=j$ is minimal with this property. Since $\sQ$ is a DVR and $\Ext^1_{\wsH}(\wS_\tau,\wS_\omega)$ is finitely generated, it follows that $\Ext^1_{\wsH}(\wS_\tau,\wS_\omega)$ is a direct sum of $m_\tau(\geq0)$ nonzero cyclic $\sQ$-modules. Let $\wY_\tau$ be the extension of $\wS_\tau^{\oplus m_\tau}$ by $\wS_\omega$,
constructed as above Lemma \ref{abstract1} (using generators for the cyclic modules). Then by Lemma \ref{abstract1}, Corollary \ref{abstract2}, and Corollary \ref{Cor5.3} we have
$$\Ext^1_{\wsH}(\wS_\tau,\wY_\tau)=0.$$

Let
$$\Omega_{j,\omega}=\{\nu\in\Omega_{j}\mid\Ext^1_{\wsH}(\wS_\nu,\wS_\omega)\neq0\}.$$
If $\nu\in\Omega_{j,\omega}\backslash\{\tau\}$, then $\Ext^1_{\wsH}(\wS_\nu,\wS_\omega)\cong\Ext^1_{\wsH}(\wS_\nu,\wY_\tau)$ by Corollary \ref{Cor5.3}, 
together with the long exact sequence
for Ext.\footnote{ We also use the fact that $f(\nu)\not=f(\tau)$ implies that $\Hom_{\sH}(\widetilde S_\nu,\widetilde S_\tau)=0$ since $\Hom_\sQ(\widetilde S_\nu,\widetilde S_\tau)$ and hence $\Hom_\sH(\widetilde S_\nu,\widetilde S_\tau)$)
are free $\sO$-modules. Thus, if $\Hom_{\sH}(\widetilde S_\nu,\widetilde S_\tau))\not=0$, then it remains nonzero
upon base change to $K$. This is impossible since $\nu$ and $\tau$ belong to different two-sided cells and $\wsH_K$ is
semisimple.}

Thus, if $\wY_{\tau,\nu}$ denotes the corresponding extension of $\wS_\nu^{\oplus m_\nu}$ by $\wY_\tau$ (again
using the construction above Lemma \ref{abstract1}), then
$$\Ext^1_{\wsH}(\wS_{\omega'},\wY_{\tau,\nu})=0\text{ for }\omega'=\tau,\nu.$$
From the general identity $\Ext^1_{\wsH}(A,C)\bigoplus\Ext^1_{\wsH}(B,C)\cong\Ext^1_{\wsH}(A\oplus B,C)$, one sees that 
$\wY_{\tau,\nu}$
 is isomorphic to the ``sum" extension of $\wS_\tau^{\oplus m_\tau}\oplus \wS_\nu^{\oplus m_\nu}$ by $\wS_\omega$.
Continuing this process, we obtain an
extension $\wY_{j}$ of $\oplus_{\tau\in\Omega_{j,\omega}}\wS_\tau^{\oplus m_\tau}$ by $\wS_\omega$, with 
 $$\Ext^1_{\wsH}(\wS_{\omega'},\wY_{j})=0\text{ for all }\omega'\in\bigcup_{\ell\leq j}\Omega_\ell .$$
Thus, $\Ext^1_{\wsH}(\wS_{\omega'},\wY_{j})\neq 0$ implies $f(\omega')>j$.

Continuing the above construction with the role of $\wS_\omega$ replaced by $\wY_{j_1}$ with $j_1=j$, 
we obtain a module $\wY_{j_1,j_2}$ such that $j_1<j_2$ and 
$$\Ext^1_{\wsH}(\wS_{\omega'},\wY_{j_1,j_2})=0\text{ for all }\omega'\in\cup_{\ell\leq j_2}\Omega_\ell.$$
Let $m$ be the maximal $f$-value. This construction will stop after a finite number $r=r(\omega)$ of steps, resulting in an $\wH$-module
$\wX_\omega:=\wY_{j_1,j_2,\cdots,j_r}$ such that 
$$f(\omega)<j_1<j_2<\cdots<j_r\leq m,\text{ and }\Ext^1_{\wsH}(\wS_{\omega'},\wX_\omega)=0\text{ for all }\omega'\in\Omega.$$

\medskip\noindent
\subsection{ A second construction of a module $\wX_\omega$.}\label{secondconstruction} The construction will generally lead to a larger module $\wX_\omega$, so is not as ``efficient" as the first construction above, in some sense. Nevertheless, the construction has similar properties, is cleaner, and has a very considerable advantage that it first builds an $\sH$-module $X_\omega$, then sets $\wX_\omega=X_{\omega,\sQ}:=
(X_\omega)_\sQ$. Both $X_\omega$ and $\widetilde X_\omega$ are built with the requirement $e\not=2$, this
condition being needed in the supporting Proposition. \ref{beforeprop}(3) below.

As before, $\Omega$ denotes the set of all left cells of $W$, and $\Omega_i=\{\omega\in\Omega\,|\,f(\omega)=i\},$
$i\in\mathbb N$. 

Fix $\omega\in\Omega$, and put $i_0=f(\omega)$. For each $i\in\mathbb Z$, put $X_{\omega,i}=0$ if $i<i_0$ (we use
these terms only as a notational convenience), and put $X_{\omega,i_0}=S_\omega$. Next, we give a recursive
defintion of $X_{\omega,j}$ for all $j\geq i_0$, with the case $j=i_0$ just given. If $X_{\omega,j}$ has been defined,
define $X_{\omega,j+1}$ as follows: Let $M$ denote the direct sum (possibly zero) of all $\sH$-modules $S_\tau$ with $f(\tau)=j+1$. Using the category $\sH$-mod for $\sC$ in the construction above Lemma \ref{abstract1}, and $Y=X_{\omega,j}$,
put $X_{\omega,j+1}=X$ in that construction (making some choice for the generators $\Ext^1_\sH(M,Y)$ that are
used). For $j$ sufficiently large, we have $\Omega_i=0$ for all $i>j$, and so $X_{\omega,i}=X_{\omega,j}$. Thus,
we set $X_\omega:=X_{\omega,j}$ for any such sufficiently large $j$.

\begin{prop}\label{beforeprop} The $\sH$-module $X_\omega$ and the increasing
 filtration $\{X_{\omega,i}\}_{i\in\mathbb Z}$ constructed above have the following properties:
\begin{itemize} \item[(1)] The smallest index of a nonzero section $X_{\omega,i}/X_{\omega,i-1}$
is $i=f(\omega)=i_0$, and the section is $S_\omega$ in that case.
\item[(2)] All sections $X_{\omega,i}/X_{\omega,i-1}$ are direct sums of modules $S_\tau$, $\tau\in\Omega$, with
varying multiplicities (possibly 0), and with $f(\tau)=i$.
\item[(3)] If $e\not=2$, $\Ext^1_{\wsH}(S_{\nu,\sQ}, X_{\omega,\sQ})=0$ for all $\nu,\omega\in\Omega$.
\end{itemize}
\end{prop}
\begin{proof}Properties (1) and (2) are immediate from the construction of $X_\omega$.

To prove (3), fix $\nu$ and $\omega\in\Omega$. We will apply Corollary \ref{Cor5.3} several times. First, it shows
the vanishing in (3) holds section by section of $X_{\omega,\sQ}$, unless $f(\nu)>f(\omega)$. So assume that
$f(\nu)>f(\omega)$.

Put $j=f(\nu)-1$ and let $M$ be the $\sH$-module used above in the construction of $X_{\omega,j+1}$ from $Y=X_{\omega,j}$. Lemma \ref{abstract1} implies the map $\Ext^1_\sH(M,Y)\to\Ext^1_\sH(M,X_{\omega,j+1})$ is the zero
map. Applying the flat base change from $\sZ$ to $\sQ$, we find that the map
$\Ext^1_{\wsH}(M_\sQ,Y_\sQ)\to\Ext^1_{\wsH}(M_\sQ,X_\sQ)$ is zero, with $X=X_{\omega,j+1}$. However,
Corollary \ref{Cor5.3} implies $\Ext^1_{\wsH}(M_\sQ,M_\sQ)=0$. Now the long exact sequence argument of
Corollary \ref{abstract2} shows that $\Ext^1_{\wsH}(M_\sQ,X_\sQ)=0$. Since $S_\nu$ is a direct summand of
$M$ (by construction, since $f(\nu)=j+1$), it follows that $\Ext^1_{\wsH}(S_{\nu,\sQ},X_\sQ)=0$. 

However, $X_\omega/X_{\omega,j+1}$ is filtered by modules $S_\tau$ with $f(\tau)>j+1=f(\nu)$.   So
$$\Ext^1_{\wsH}(S_{\nu,\sQ},(X_{\omega}/X_{\omega,j+1})_\sQ)=0 $$
by Corollary \ref{abstract2} again. Together
with the conclusion of the previous paragraph, this gives the required vanishing $\Ext^1_{\wsH}(S_{\nu,\sQ},X_{\omega,\sQ})=0$.
\end{proof}

To complete the second construction, set $\widetilde X_\omega=X_{\omega,\sQ}$.

\subsection{The main result}

Let $\Omega'$ be the set of all left cells that do not contain the longest element of a parabolic subgroup. Put
$$\wT=\bigoplus_{\lambda \subseteq S}x_\lambda\wH\quad\text{ and }\quad\widetilde\sX=\bigoplus_{\omega\in\Omega'}\wX_\omega.$$
Here and in the theorem below, objects (modules, algebras, etc.) are decorated with a tilde $\sim$ because they
are taken over the algebra $\sQ$ in (\ref{modular system}).

We are now ready to prove the following main result of the paper. 

\begin{thm} \label{4.4}Assume that $e\not=2$. Let $\wT^+=\wT\oplus \widetilde\sX$, $\wsA^+=\End_{\wH}(\wT^+)$ and $\wDelta(\omega)=\Hom_{\wH}(\wS_{\omega},\wT^+)$ for $\omega\in\Omega$. Then $\{\wDelta(\omega)\}_{\omega\in\Omega}$ is a strict stratifying system for the category $\wsA^+$-{ mod} with respect to the quasi-poset $(\Omega,\leq_f)$.
\end{thm}

\begin{proof} For each left cell $\omega$, put $\wwT_\omega=x_\la\wH$ if $\omega$ contains the longest element $w_{\lambda,0}$ of $W_\la$, where $\la\subseteq S$. If there is no such $\la$ for $\omega$, put $\wwT_\omega=\wX_\omega$ as constructed in \S5.2. (One can use the $\wX_\omega$ from \S5.3 with slight adjustments, left to the
reader.) In the first case, $\wwT_\omega$ has a filtration by dual left cell modules, and $\wS_\omega$ appears at the bottom. Moreover, $f(\omega)<f(\omega')$ for any other filtration section $\wS_{\omega'}$, by Lemma~\ref{strict}. This same property holds also in the case $\wwT_\omega=\wX_\omega$ by construction.

Put $\wwT=\oplus_\omega\wwT_\omega$ and note $\wT^+=\wwT$. We will apply Theorem \ref{strattheorem} to $\wwT$ and the various $\wwT_\omega$, where $\wsH$ plays the role of the algebra $B$ there, $\sQ$ plays the role of
$R$ there, $\wS_\omega$ is $S_\lambda$, etc. 
We are required to the check three conditions (1), (2), (3) in Theorem \ref{strattheorem}.  The construction in
\S5.2 of dual left cell filtrations of the various $\wwT_\omega$ is precisely what is required for the verification of (1).

Condition (2) translates directly to the requirement
$$\Hom_\wH(\wS_\mu,\wwT_\omega)\neq0\implies \omega\leq_f\mu$$
for given $\mu,\omega$. However, if $\Hom_\wH(\wS_\mu,\wwT_\omega)\not=0$, there must be a nonzero $\Hom_\wH(\wS_\mu,\wS_{\omega'})$ for some filtration section $\wS_{\omega'}$ of $\wwT_\omega$. In particular, $f(\omega')\geq f(\omega)$. Also,
$(\wS_\mu)_K$ and $(\wS_{\omega'})_K$ must have a common irreducible constituent, forcing the two-sided cells containing $\mu$ and $\omega'$ to agree. This gives $f(\mu)=f(\omega')\geq f(\omega)$; so (2) holds.

Finally, 
\begin{equation}\label{Finally}\Ext_\wH^1(\wS_\mu,\wwT_\omega)=0\text{ for all }\mu,\omega.\end{equation}
This follows from the construction \S5.2 for $\wwT_\omega=\wX_\omega$ and by Corollary  \ref{vanishing} in case $\wwT_\omega =x_\la\wH$. The conclusion of Theorem \ref{strattheorem} now immediately gives the theorem we are proving here.
\end{proof}

\section{Identification of $\wsA^+=\End_{\wH}(\wT^+)$.}

The constructions in \S\S5.2, 5.3 of the modules $\wX_\omega$ in the previous section work just
 as well using the modules $\wS_E:=\wS(E)^*$ for $E\in\irr({\mathbb Q}W)$ defined in (\ref{Spechtdual}) to replace the dual left cell modules $\widetilde S_\omega$. This results in right $\sH$-modules $\wX_E$. As in the case of $\wX_\omega$, we have the following property, with the same proof. In the statement of the following proposition, $\wX_E$ can be defined using either of the two constructions.

\begin{prop} Assume that $e\not=2$. Then $\Ext_{\wH}^1(\wS_{E'},\wX_E)=0$ for all $E,E'\in\text{\rm Irr}({\mathbb Q}W)$. 
\end{prop}

If we use the first construction given in \S5.2,  the modules $\wX_E$ have strong indecomposability properties, which the modules $\wX_\omega$, $\omega\in\Omega$ generally do not have with either construction.
In the following proposition, we assume that $\wX_E$ is defined by the first construction \S5.2.  

The following result  can be argued without using RDAHAs, but it is faster to quote Rouquier's 1-faithful covering theory, especially 
\cite[Thm.~5.3]{Ro08}, which applies to our $e\neq2$ case, over $\sR$, where 
$$\sR:=(\bbC[t,t^{-1}]_{(t-\sqrt{\zeta})})^\wedge$$
 is the completion of the localization $\bbC[t,t^{-1}]_{(t-\sqrt{\zeta})}$ at the maximal ideal $(t-\sqrt{\zeta})$. Note that $\sR$ is a $\sQ$-module via the natural ring homomorphism $\sQ\to\sR$.
 Note also that the set $\irr({\mathbb Q}W)$ 
 corresponds naturally to the set $\irr(W):=\irr({\mathbb C}W)$ in \cite{Ro08}.

\begin{prop} Assume that $e\not=2$. The right $\wsH$-modules $\wX_E$ are indecomposable, as is each $\wX_E\otimes k$. The endomorphism algebras of all these modules are local with radical quotient $k$.
\end{prop}

\begin{proof}
It is clear that $\wX_{E,\sR}=\wX_E\otimes_\sQ\sR$ can be constructed from $\wS_{E,\sR}$ in the same way that $\wX_E$ 
is constructed from $\wS_E$, again using the method of \S5.2. Also, the proof of \cite[Thm. 6.8]{Ro08} shows that the $\sR$-dual of $\widetilde S_{E,\sR}$ is the KZ-image
of the standard module $\Delta_\sR(E)$ in the $\sR$-version of $\sO$. (Recall the issues in ftn.~2.)

Consequently, by the 1-faithful property,  $(\wX_{E,\sR})^*$ is the image of a dually constructed module $P$ under the
functor KZ, filtered by standard modules, and with $\Ext^1_{\sO}(P,-)$ vanishing on all standard modules. Such a module $P$ is projective in $\sO$, by \cite[Lem.~4.22]{Ro08}. (We remark that both $\sO$ and KZ would be given a subscript $\sR$ in \cite{GGOR03} though not in \cite{Ro08}.)

If we knew $P$ were indecomposable, we could say $\wX_{E,\sR}$ is indecomposable. However, the indecomposability
of $P$ requires proof.\footnote{A similar point should be made regarding the uniqueness claim in \cite[Prop.~4.45]{Ro08}, which is false without a minimality assumption on $Y(M)$ there.} We do this by showing $P$ is the projective cover in $\sO$ of the standard module $\Delta(E)=\Delta_\sO(E)$. We can, instead, inductively show the truncation $P_i$, associated to the poset ideal of all $E'\in\text{Irr}({\mathbb Q}W)$ with $f(E')\leq i$, is the projective cover of $\Delta(E)$ in the associated truncation $\sO_i$ of $\sO$. This requires $\Delta(E)$ to be an object of $\sO_i$, or equivalently $f(E)\leq i$. 

If $f(E)=i$, then $P_i=\Delta(E)$ is trivially the projective cover of $\Delta(E)$. Inductively, $P_{i-1}$ is the projective cover of $\Delta(E)$ in $\sO_{i-1}$ for some $i>f(E)$. Let $P'$ denote the projective cover of $\Delta(E)$ in $\sO_i$. The truncation $(P')_{i-1}$ to $\sO_{i-1}$ of $P'$ --- that is, its largest quotient which is an object of $\sO_{i-1}$ --- is clearly isomorphic to $P_{i-1}$. Let $\varphi:P'\to P_{i}$ be a homomorphism extending a given isomorphism $\psi:(P')_{i-1}\to P_{i-1}$ and let $\tau:P_i\to P'$ be a homomorphism extending $\psi^{-1}$. Let $M,M'$ denote the kernels of the natural surjections $P_i\twoheadrightarrow P_{i-1}$ and $P'\twoheadrightarrow (P')_{i-1}$. 
The map $\tau\varphi:P'\to P'$ is surjective and, consequently, it is an isomorphism. It induces the identity on $(P')_{i-1}$. Therefore, the induced map
$$\tau|_M\varphi|_{M'}:M'\longrightarrow M'$$
 is an isomorphism, and $M=M'\oplus M''$ for some object $M''$ in $\sO$. By construction, $M$ is a direct sum of objects $\Delta(E')$, with $f(E')=i$, each appearing with multiplicity $m_{E'}=\text{rank}(\Ext^1_\sO(P_i,\Delta(E')))$. However, 
$$\Ext^1_\sO(P_{i-1},\Delta(E'))\cong\Hom_\sO(M',\Delta(E')).$$ 
It follows that $M''=0$ and $P_i\cong P'$ is indecomposable.

In particular, $P$ is indecomposable and consequently $\wX_{E,\sR}$ is indecomposable, as noted. In turn, this implies $\wX_E$ is indecomposable. The 0-faithfulness (or just the covering property itself) of the cover given by $\sO$ and KZ imply 
$$\End_{\wH_\sR}(\wX_{E,\sR})^{\text{op}}\cong \End_{\wH_\sR}(\wX_{E,\sR}^*)^{\text{op}}\cong\End_\sO(P).$$
Thus, the base-changed module $P\otimes_\sR \bbC$  has endomorphism ring
$$\End_{\sO_\bbC}(P\otimes_\sR \bbC)\cong\End_{\sO}(P)\otimes_\sR \bbC,$$
where $\sO_\bbC$ is the $\bbC$-version of $\sO$. This is a standard consequence of the projectivity of $P$. By \cite[Thm.~5.3]{Ro08}, the $\bbC$ versions of KZ and $\sO$ give a cover for $\wH_\sR\otimes\bbC$. So $\End_{\wH_\bbC}(\wX_{E,\sR}\otimes \bbC)^{\text{op}}\cong\End_{\sO_\bbC}(P\otimes \bbC)$ is local, with radical quotient $\bbC$.

However, we have
$$(\wX_E\otimes_{\sQ}k)\otimes_k\bbC\cong\wX_{E,\sR}\otimes\bbC.$$
In particular, $\wX_E\otimes_{\sQ}k$ is indecomposable since (by endomorphism ring considerations) the $\wH_\sR\otimes\bbC$-module $\wX_{E,\sR}\otimes\bbC$ is indecomposable. So the endomorphism ring of $\wX_E\otimes_{\sQ}k$ over the finite dimensional algebra $\wH\otimes_\sQ k$ is local. The radical quotient is a division algebra $D$ over $k$ with base change $-\otimes_k\bbC$ to a semisimple quotient of 
$\End_{\wH_\bbC}(\wX_{E,\sR}\otimes \bbC)$, which could only be $\bbC$ itself. Consequently, $D=k$. 

Finally, the vanishing $\Ext_\wH^1(\wX_E,\wX_E)=0$ implies 
$$\End_\wH(\wX_E)\otimes_\sQ k\cong\End_{\wH_k}(\wX_E\otimes_\sQ k).$$
So the ring $\End_\wH(\wX_E)$ is local with radical quotient $k$. This completes the proof.
\end{proof}

\begin{lem} Assume $e\not=2$.
Let $E\in\text{\rm Irr}(\bbQ W)$. Then $\wX_E$ is a direct summand of $\wT^+$.
\end{lem}
\begin{proof}
Suppose first $\wS(E)$ is a direct summand of a left cell module $\wS(\omega)=:\wS^\omega\cong(\wS_\omega)^*$ where $\omega$ contains the longest element of a parabolic subgroup $W_\lambda$, $\lambda\subseteq S$. This implies $\wS_\omega$ is the lowest term in the dual left cell module filtration of $x_\la\wH$. Consequently, there is an inclusion $\psi:\wS_E\to x_\la\wH$ with cokernel filtered by (sections)
$\wS_{E'}$, $E'\in\text{\rm Irr}({\mathbb Q}W)$. Thus, $\psi^{-1}:\psi(\wS_E)\to\wX_E$ may be extended to a map $\phi:x_\la\wH\to\wX_E$ of $\wH$-modules. Similarly (using $e\neq2$ and Corollary \ref{vanishing}), there is a map $\tau:\wX_E\to x_\la\wH$ extending $\psi$. The composite $\tau\phi$ restricts to the identity on $\wS_E\subseteq\wX_E$.

On the other hand, restriction from $\wX_E$ to $\wS_E$ defines a homomorphism
$$\End_{\wH}(\wX_E)\longrightarrow \End_{\wH}(\wS_E)$$
since $(\wS_E)_K$ is a unique summand of the (completely reducible) $\wH\otimes_\sQ K$-module $\wX_E\otimes_\sQ K$. (Observe $\wS_E=\wX_E\cap(\wS_E)_K$, since the $\sQ$-torsion module $(\wX_E\cap(\wS_E)_K)/\wS_E$ must be zero in the $\sQ$-torsion free module $\wX_E/\wS_E$.)
Thus, $\tau\phi$ is a unit in the local endomorphism ring $\End_\wH(\wX_E)$, so $\wX_E$ is a summand of $x_\la\wH$, and hence of $\wT$.

Next consider the case in which $\wS_E$ is a summand of a dual left cell module $\wS_\omega$ (this always happens for some $\omega$), but $\omega$ does not contain the longest element of any parabolic subgroup. In this case, $\wX_\omega$ is one of the summands of $\widetilde\sX$ by construction. The argument above may be repeated with $\wX_\omega$ playing the role of $x_\la\wH$. In the same way, $\wX_E$ is a direct summand of $\wX_\omega$, and thus of $\widetilde\sX$.

In both cases, we conclude that $\wX_E$ is a direct summand of $\wT\oplus\widetilde\sX=\wT^+$.
\end{proof}

\begin{thm}\label{cases} Assume that $e\not=2$.
The $\sQ$-algebra $\wsA^+$ is quasi-hereditary, with standard modules $\wDelta(E)=\Hom_\wH(\wS_E,\wT^+)$, $E\in\text{\rm Irr}({\mathbb Q}W)$, and partial order $<_f$.
\end{thm}

\begin{proof}
We have already seen that this algebra is standardly stratified with strict stratifying system $\{\wDelta(\omega)\}_{\omega\in\Omega}$.
Clearly, $\wDelta(\omega)$ is a direct sum of various $\wDelta(E)$'s, and every $\wDelta(E)$ arises as such a summand.

Put $\wP(E)=(\wX_E)^\diamond:=\Hom_\wH(\wX_E,\wT^+)$, $E\in\text{Irr}({\mathbb Q}W)$.
Then $\wP(E)$ is a direct summand of $\wsA^+=\End_\wH(\wT^+)$, viewed as a left module over itself. Thus, $\wP(E)$ is projective as an $\wsA^+$-module, and $\wP(E)^\diamond:=\Hom_{\wsA^+}(\wP(E),\wT^+)$ is naturally isomorphic to $\wX_E$. In particular, the contravariant functor $(-)^\diamond$ gives an isomorphism
$$\End_{\wsA^+}(\wP(E))\cong (\End_\wH(\wX_E))^\text{op}.$$
Consequently, $\wP(E)$ also has a local endomorphism ring with radical quotient $k$, as does $\End_{\wsA^+_k}(\wP(E)\otimes_\sQ k)$. It follows that $\wP(E)$ is an indecomposable projective $\wsA^+$-module with a irreducible head. (The arguments in this paragraph are largely standard, many taken from \cite{DPS98a}.)

By (\ref{Finally}), $\Ext^1_\wH(\wS_\omega,\wT^+)=0$ for all dual left cell module $\wS_\omega$. Consequently, a similar vanishing holds with $\wS_\omega$ replaced by any module $\wS_{E'}$, $E'\in\text{Irr}({\mathbb Q}W)$. It follows that the restriction map
$$\wP(E)=\Hom_{\wH}(\wX_E,\wT^+)\longrightarrow \Hom_\wH(\wS_E,\wT^+)=\wDelta(E)$$
is surjective. Hence, $\wDelta(E)$ has an irreducible head. Also, repeating the argument for filtered submodules of $\wX_E$, we find that the kernel of the above map  has a filtration with sections $\wDelta(E')$, $E'\in\text{Irr}({\mathbb Q}W)$ (rather than $\wX_E$ itself), satisfying $f(E')>f(E)$.

Next, we claim that $\wDelta(E)^\diamond:=\Hom_{\wsA^+}(\wDelta(E),\wT^+)$ is naturally isomorphic to $\wS_E$. More precisely, we claim that the natural map $\wS_E\overset{\text{ev}}\to(\wS_E)^{\diamond\diamond}$ is an isomorphism. We showed above that the sequence
$$0 \longrightarrow (\wX_E/\wS_E)^\diamond\longrightarrow(\wX_E)^\diamond\longrightarrow(\wS_E)^\diamond\longrightarrow0$$
is exact. Applying $(-)^\diamond$ once more, we get an injection
$$0\longrightarrow(\wS_E)^{\diamond\diamond}\longrightarrow(\wX_E)^{\diamond\diamond}$$
with $\wX_E\overset{\text{ev}}\to(\wX_E)^{\diamond\diamond}$ an isomorphism. This gives inclusions
$$\wS_E\cong\text{ev}(\wS_E)\subseteq(\wS_E)^{\diamond\diamond}\subseteq (\wX_E)^{\diamond\diamond}\cong\wX_E.$$
If $(-)\otimes_\sQ K$ is applied, the first inclusion becomes an isomorphism. This gives 
$$(\wS_E)^{\diamond\diamond}\subseteq (\wX_E)^{\diamond\diamond}\cap(\wS_E)_K=\wS_E$$
identifying $\wX_E$ with $(\wX_E)^{\diamond\diamond}$ and $\wS_E$ with its image in $(\wX_E)^{\diamond\diamond}$. Consequently,
ev$(\wS_E)=(\wS_E)^{\diamond\diamond}$, proving the claim.

Finally, we suppose $E\not\cong E'\in\text{Irr}({\mathbb Q}W)$ and $\Hom_{\wsA^+}(\wP(E'),\wDelta(E))\neq0$. Using the identifications $\wP(E')=(\wX_{E'})^\diamond$,
$\wDelta(E)=(\wS_E)^\diamond$, $\wP(E')^\diamond\cong \wX_{E'}$, and $\wDelta(E)^\diamond\cong \wS_E$, we have
$$0\neq\Hom_{\wsA^+}(\wP(E'),\wDelta(E))\cong\Hom_\wH(\wS_E,\wX_{E'})\subseteq \Hom_{\wH_K}(\wS_E\otimes_\sQ K,\wX_{E'}\otimes_\sQ K).$$
This implies $f(E')<f(E)$. It follows now from \cite[Thm. 1.2.8]{DPS98a} (in the context of stratified algebras),
 \cite[Cor. 2.5]{DS94}, or \cite[Thm.~4.16]{Ro08} that $\wsA^+$ is quasi-hereditary over $\sQ$.
\end{proof}

We are now ready to establish the category equivalence mentioned in the introduction. Again, we use the covering
theory of \cite{Ro08}.

\begin{thm}\label{cor6.5} Assume that $e\not=2$.
The category of left modules over the base-changed algebra
$$\wsA^+_\sR:=\wsA^+\otimes_\sQ\sR$$
is equivalent to the $\sR$-category $\sO$ of modules, as defined in \cite{Ro08} for the RDAHA  associated to $W$ over $\sR$.
\end{thm}

\begin{proof}
Continuing the proof of the theorem above, the projective indecomposable $\wsA^+$-modules are the various $\wP(E)=(\wX_E)^\diamond$. Consequently, $\wT^+=(\wsA^+)^\diamond$ is the direct sum of the modules, $\wX_E$, each with nonzero multiplicities. The modules
$\wX_{E,\sR}$ remain indecomposable, as observed in the proof of the indecomposability of the modules $\wX_E$ above. By construction, $\Ext^1_\wH(\wS_{E'},\wX_E)=0$ for all $E,E'\in\text{Irr}({\mathbb Q}W)$. Thus, there is a similar vanishing for $\wS_{E',\sR}$ and $\wX_{E,\sR}$, and---in the reverse order---for their $\sR$-linear duals.
Observe that $(\wS_{E',\sR})^*\cong \wS(E')\otimes_\sQ \sR$ is KZ$(\Delta(E'))$, taking $\Delta(E')=\Delta_{\sO}(E')$ to be the standard module for the category $\sO$ over
$\sR$ as discussed in \cite{Ro08} together with KZ for this category.

Put 
$$Y=\bigoplus_E(\wX_{E,\sR})^*$$ 
and set $Y(\wS_{E,\sR}^*)=(\wX_{E,\sR})^*$. This notation imitates that of \cite[Prop.~4.45]{Ro08}.  The first part of this proposition is missing a necessary minimality assumption on the rank of $Y(M)$, in the terminology there.\footnote{The proposition claims uniqueness for a pair $(Y(M), p_m)$. However, one gets another pair by adding a direct
summand $F(P)$ to the kernel of $p_m$, where $P$ is any finitely generated module  in the highest weight category
 $\sC$.} However, this is satisfied for $M=(\wS_{E,\sR})^*$ and $Y(M)=(\wX_{E,\sR})^*$ because $(\wX_{E,\sR})^*$ is indecomposable. 
Several other corrections, in addition to the minimality requirement, should be made to \cite[Prop.~4.45]{Ro08}: 
\begin{itemize}
\item $A'$ should be redefined as $\End_B(Y)^{\text{op}}$; 
\item $P'$ should be redefined as $\Hom_B(Y,B)^{\text{op}}$.
\end{itemize}
 In addition, $B$ in \cite[4.2.1]{Ro08} should be redefined as $\End_A(P)^{\text{op}}$. The ``op''s here and above insure action on the left, and consistency with \cite[Thm.~5.15, Thm.~5.15]{GGOR03}. The definition of $P'$ is given to be consistent with the basis covering property $\End_{A'}(P')^{\text{op}}\cong B$, as in \cite[Thm.~5.15]{GGOR03}---we do not need this fact below.

 With these changes,  \cite[Thm.~5.3, Prop.~4.45, Cor.~4.46]{Ro08} guarantees that  $\sA'$-{ mod} is equivalent to $\sO$, where $\sA'=\End_{\wH_\sR}(Y)$. (All we really need for this are the 0- and 1-faithfulness of the $\sO$ version of the KZ functor.)
 However,  $\End_{\wH_\sR}(Y)\cong\End_{\wH_\sR}(Y^*)^{\text{op}}$, and $Y^*$ is the direct sum $\oplus_E\wX_{E,\sR}$. Hence,
 $$Y^{*\diamond}\cong\bigoplus_E(\wX_{E,\sR})^\diamond\cong\bigoplus_E\wP(E)\otimes_\sQ \sR.$$
 
 Recall that $(\wX_{E,\sR})^{\diamond\diamond}\cong\wX_{E,\sR}$, so that the analogous property holds for $Y^*$.
 Thus, $\End_{\wH_\sR}(Y^*)^{\text{op}}\cong\End_{\wsA^+_\sR}(Y^{*\diamond})$. Since the module $Y^{*\diamond}$
  as displayed above is clearly a projective generator for $\wsA_\sR^+$, there is a Morita equivalence over $\sR$ of $\wsA_\sR^+$ with $\sA'$. Hence, $\wsA_\sR^+$-$\text{mod}$ is equivalent to $\sO$, as $\sR$-categories.
\end{proof}

\section{Appendix: comparison with \cite[Conj. 2.5.2]{DPS98a} }
    Conjecture \ref{conjecture} in this paper retains the most essential features of \cite[Conj. 2.5.2]{DPS98a}, but is more flexible. In particular,

\begin{itemize}
\item[(1)]  Conjecture \ref{conjecture} does not specify the preorder $\leq$, only requiring that it be strictly compatible with the partition of $\Omega$ into two-sided cells. This allows the use of the preorder $\leq_f$, defined in \S2 above. \cite[Conj. 2.5.2]{DPS98a} specifies for $\leq$ the preorder $\leq^{\text{\rm op}}_{LR}$ built from the preorder $\leq_{LR}$ originally used by Kazhdan-Lusztig to define the two-sided cells. In both cases, the set $\overline \Omega$ of ``strata" is the same, identifying with the set of two sided cells. 
\smallskip
    
    \item[(2)] Conjecture \ref{conjecture}  concerns the Hecke algebra $\sH$ (defined by the relations (\ref{relations}) over $\sZ={\mathbb Z}[t,t^{-1}]$, whereas \cite[Conj. 2.5.2]{DPS98a} uses 
    Hecke algebras over ${\mathbb Z}[t^2,t^{-2}]$. Largely, this change has been made to conform to the literature, which most often uses the former ring. There is an additional advantage that the quotient field ${\mathbb Q}(t)$ is almost always a splitting field
    for the Hecke algebra $\sH_{{\mathbb Q}(t)}$.\footnote{${\mathbb Q}(t)$ is always a splitting field in case the rank is greater than 2. In the rank 2 case of $^2F_4$, $\sH_{{\mathbb Q}(t)}$ splits after  
    $\sqrt{2}$ is adjoined. The conjecture in all rank 2 cases follows from  \cite[\S3.5]{DPS98a}. }

\item[(3)] The role of $\sA^+_R$ in Conjecture \ref{conjecture} is played by $\End_{\sH_R}(\sT^+_R)$ in \cite[Conj. 2.5.2]{DPS98a}. The two $R$-algebras are the same whenever $R$ is flat over $\sZ={\mathbb Z}[t,t^{-1}]$. While it is an interesting question as to whether or not such a base change property holds for any $\sZ$-algebra $R$, it seems best to separate this issue from the main stratification proposal of the conjecture.
\end{itemize}

\medskip

Finally, we mention that the original conjecture \cite[Conj. 2.5.2]{DPS98a}
 was checked in that paper for all rank two types (in both the equal and unequal parameter cases), and checked later in type $A$ for all ranks; see \cite{DPS98a}. These verifications show also that Conjecture \ref{conjecture} is true in these cases.

\end{document}